\documentclass{amsart}

\usepackage{amssymb} 

\newtheorem{theorem}{Theorem}[section]
\newtheorem{lemma}[theorem]{Lemma}
\newtheorem{corollary}[theorem]{Corollary}

\theoremstyle{definition}
\newtheorem{definition}[theorem]{Definition}

\theoremstyle{remark}

\numberwithin{equation}{section}

\begin{document}


\title{An incomplete Riemann Zeta function as a fractional integral}

\author{S. M. Crider}
\address{}

\author{S. Hillstrom}
\address{}

\date{September 26, 2024}

\dedicatory{The authors thank Drs. Hong-Ming Yin, Christopher Bishop, and Samuel Johnston.}

\begin{abstract}
    An incomplete Riemann zeta function can be expressed as a lower-bounded, improper Riemann-Liouville fractional integral, which, when evaluated at $0$, is equivalent to the complete Riemann zeta function. Solutions to Landau’s problem with $\zeta(s) = \eta(s)/0$ establish a functional relationship between the Riemann zeta function and the Dirichlet eta function, which can be represented as an integral for the positive complex half-plane, excluding the pole at $s = 1$. This integral can be related to a lower-bounded Riemann-Liouville fractional integral directly via Cauchy’s Formula for repeated integration extended to the complex plane with improper bounds. In order to establish this relationship, however, specific existence conditions must be met. The incomplete Riemann zeta function as a fractional integral has some unique properties that other representations lack: First, it obeys the semigroup property of fractional integrals; second, it allows for an additional functional relationship to itself through differentiation in other regions of convergence for its fractional integral representation. The authors suggest development of the Riemann zeta function using this representation and its properties.
\end{abstract}

\maketitle

\section{Introduction}\label{sec1}

    It is well-known that the Dirichlet $\eta$ function given by
    \begin{equation}
        \eta(s)=\sum_{n=1}^{\infty}\frac{(-1)^{n-1}}{n^{s}}
    \end{equation}
    can be written in relation to the Riemann $\zeta$ function, given by
    \begin{equation}
        \zeta(s)=\sum_{n=1}^{\infty}\frac{1}{n^{s}}
    \end{equation}
    for $\Re(s)>1$ \cite{Titchmarsh}. This relation is given as
    \begin{equation}\label{Eta Equals zeta}
        \eta(s) = (1 - {2}^{1-s})\zeta(s).
    \end{equation}
    Based on work from Landau, it is also now well-known that the relation above can be rewritten in terms of $\zeta$ instead of $\eta$ for $\Re(s) > 0$ except at $s = 1$ \cite{Landau, Sondow, Widder}. The relation holds
    \begin{equation}\label{Zeta Equals eta}
        \zeta(s) = \frac{\eta(s)}{1 - {2}^{1-s}} = \frac{{2}^{s}}{{2}^{s} - {2}}\eta(s).
    \end{equation}
    From this, the Riemann $\zeta$ function is defined using the integral representation of the Dirichlet $\eta$ function for all $\Re(s) > 0$ except at $s = 1$ \cite{Spiegel}:
    \begin{equation}\label{Zeta as eta Integral}
        \zeta(s) = \frac{2^{s}}{(2^{s} - 2) \Gamma(s)} \int_{0}^{\infty} \frac{t^{s - 1}}{e^{t} + 1} \,\mathrm{d}t
    \end{equation}
    where the $\Gamma$ function is defined as
    \begin{equation}
        \Gamma(s)=\int_{0}^{\infty}e^{-t}t^{s-1}\,\mathrm{d}t.
    \end{equation}
    First, we define an incomplete $\eta$ function \cite{Dunster, Kölbig1, Kölbig2}
    \begin{equation}\label{Incomplete eta}
        \eta(s, x) = \frac{1}{\Gamma(s)} \int_{x}^{\infty} \frac{(t - x)^{s - 1}}{e^{t} + 1} \,\mathrm{d}t, \quad x \in [0, \infty)
    \end{equation}
    such that $\eta(s, 0) = \eta(s)$. Using this representation of $\eta$, we can further define an incomplete $\zeta$ function \cite{Dunster, Kölbig1, Kölbig2}
    \begin{equation}\label{Incomplete zeta}
        \zeta(s, x) = \frac{2^{s}}{2^{s} - 2} \eta(s, x) = \frac{2^{s}}{(2^{s} - 2) \Gamma(s)} \int_{x}^{\infty} \frac{(t - x)^{s - 1}}{e^{t} + 1} \,\mathrm{d}t, \quad x \in [0, \infty)
    \end{equation}
    such that $\zeta(s, 0) = \zeta(s)$. Finally, it is important to note that the general form of a lower-bounded Riemann-Liouville fractional integral of order $s$ with $\Re(s) > 0$ is represented by the operator $_{a}\,\mathrm{I}\,_{x}^{s}$. This operator on some locally integrable function $f$ gives the equation
    \begin{equation}\label{Lower RL Integral}
        \left(_{a}\,\mathrm{I}\,_{x}^{s}f\right)(x) = \frac{1}{\Gamma(s)} \int_{a}^{x} (x - t)^{s - 1}f(t) \,\mathrm{d}t,
    \end{equation}
    which follows directly from the extension of Cauchy's Formula for Repeated Integration to the positive complex half-plane \cite{Cauchy, Zhou}. The authors note the similarity between Equations \ref{Incomplete eta} and \ref{Lower RL Integral}. Using these definitions, the authors show it is possible to redefine this incomplete Riemann $\zeta$ function using Riemann-Liouville fractional integrals. This new definition and its implications are discussed in the sections following.

\section{Cauchy's Formula for Repeated Integration with improper limits of integration}\label{sec2}

    Cauchy's Formula for Repeated Integration asserts that $n$ integrals of some continuous function $f$ can be condensed into one integration by the equation
    \begin{equation}\label{Cauchy's Formula}
        f^{(-n)}(x) = \frac{1}{(n-1)!} \int_{a}^{x} (x - t)^{n - 1}f(t) \,\mathrm{d}t.
    \end{equation}
    The lower bound $a$ is generally constant, but it can be unbounded under specific conditions \cite{Apostol, Cauchy}.
    
    \begin{lemma} 
    \label{Cauchy's Formula with Improper Bounds}
    
        Let $f$ be a continuous, real-valued function on $(-\infty, b]$. Then the $n$th repeated integral of $f$ over $(-\infty, x)$ such that $x \in (-\infty, b]$ is given by the single integration
        $$f^{(-n)}(x) = \frac{1}{(n-1)!} \int_{-\infty}^{x} (x - t)^{n - 1}f(t) \,\mathrm{d}t$$
        if the integral converges $\forall x \in (-\infty, b]$.
        
    \end{lemma}
    
        \begin{proof}
        
            Let $f$ be a continuous, real-valued function on $(-\infty, b]$. Then by Equation \ref{Cauchy's Formula},
            $$f^{(-1)}(x) = \frac{1}{0!} \int_{-\infty}^{x} (x - t)^{0}f(t) \,\mathrm{d}t = \int_{-\infty}^{x} f(t) \,\mathrm{d}t = F(x)$$
            by the Fundamental Theorem of Calculus if the integral converges $\forall x \in (-\infty, b]$. Assume this is true for $n$. Then if
            $$\int_{-\infty}^{x} (x - t)^{n}f(t) \,\mathrm{d}t$$
            is convergent $\forall x \in (-\infty, b]$, it follows that
            $$\frac{\mathrm{d}}{\mathrm{d}x} \left[ \frac{1}{n!} \int_{-\infty}^{x} (x - t)^{n}f(t) \,\mathrm{d}t \right]$$
            $$= \frac{1}{n!} \cdot \frac{\mathrm{d}}{\mathrm{d}x} \left[ \lim_{a \to -\infty} \int_{a}^{x} (x - t)^{n}f(t) \,\mathrm{d}t \right]$$
            $$= \frac{1}{n!} \cdot \lim_{a \to -\infty} \left[ \frac{\mathrm{d}}{\mathrm{d}x} \int_{a}^{x} (x - t)^{n}f(t) \,\mathrm{d}t \right]$$
            $$= \frac{1}{n!} \cdot \lim_{a \to -\infty} \left[ (x - x)^{n}f(x) \cdot \frac{\mathrm{d}}{\mathrm{d}x}(x) - (a - t)^{n}f(a) \cdot \frac{\mathrm{d}}{\mathrm{d}x}(a) + \int_{a}^{x} \frac{\mathrm{\partial}}{\mathrm{\partial}{x}} (x - t)^{n}f(t) \,\mathrm{d}t \right]$$
            $$= \frac{1}{n!} \int_{-\infty}^{x} \frac{\mathrm{\partial}}{\mathrm{\partial}{x}} (x - t)^{n}f(t) \,\mathrm{d}t$$
            $$= \frac{1}{(n - 1)!} \int_{-\infty}^{x} (x - t)^{n - 1}f(t) \,\mathrm{d}t$$
            by the Leibniz Integral Rule. Then, applying the induction hypothesis,
            $$f^{-(n + 1)}(x) = \int_{-\infty}^{x} \int_{-\infty}^{\sigma_{1}} \cdots \int_{-\infty}^{\sigma_{n}} f(t) \,\mathrm{d}t \ldots \mathrm{d}\sigma_{2} \,\mathrm{d}\sigma_{1}$$
            $$= \int_{-\infty}^{x} \frac{1}{(n - 1)!} \int_{-\infty}^{\sigma_{1}} (\sigma_{1} - t)^{n - 1}f(t) \,\mathrm{d}t \,\mathrm{d}\sigma_{1}$$
            if $f^{-(n)}(x)$ is convergent $\forall x \in (-\infty, b]$. By extension of the Fundamental Theorem of Calculus, it further follows that
            $$\int_{-\infty}^{x} \frac{1}{(n - 1)!} \int_{-\infty}^{\sigma_{1}} (\sigma_{1} - t)^{n - 1}f(t) \,\mathrm{d}t \,\mathrm{d}\sigma_{1}$$
            $$= \int_{-\infty}^{x} \frac{\mathrm{d}}{\mathrm{d}\sigma_{1}} \left[ \frac{1}{n!} \int_{-\infty}^{\sigma_{1}} (\sigma_{1} - t)^{n}f(t) \,\mathrm{d}t \right] \,\mathrm{d}\sigma_{1}$$
            $$= \frac{1}{n!} \int_{-\infty}^{x} (x - t)^{n}f(t) \,\mathrm{d}t.$$
            
        \end{proof}

    The improper form of Cauchy's Formula given in Lemma \ref{Cauchy's Formula with Improper Bounds} can be generalized from $n \in \mathbb{N}^{+}$ to $s \in \mathbb{C}^{+}$. This generalization is used to define the Riemann-Liouville integral with improper lower bound \cite{Zhou}.

    \begin{corollary} 
    \label{Cauchy's Formula Extended using RLI}
    
        For $\Re(s)>0$, a Riemann-Liouville fractional integral of order $s$ with improper lower bound is defined as
        $$\left(_{-\infty}\,\mathrm{I}\,_{x}^{s}f\right)(x) = \frac{1}{\Gamma(s)} \int_{-\infty}^{x} (x - t)^{s - 1}f(t) \,\mathrm{d}t$$
        given that all conditions of Lemma \ref{Cauchy's Formula with Improper Bounds} are met.
        
    \end{corollary}

        \begin{proof}
        
            From Equation \ref{Lower RL Integral}, the lower-bounded Riemann-Liouville fractional integral of order $s$ of a locally integrable function $f$ is given by
            $$\left(_{a}\,\mathrm{I}\,_{x}^{s}f\right)(x) = \frac{1}{\Gamma(s)} \int_{a}^{x} (x - t)^{s - 1}f(t) \,\mathrm{d}t$$
            in the half-plane $\Re(s)>0$. Taking the limit as $a \to -\infty$ of the right-hand side of the equation above yields
            $$\lim_{a \to -\infty} \frac{1}{\Gamma(s)} \int_{a}^{x} (x - t)^{s - 1}f(t) \,\mathrm{d}t = \frac{1}{\Gamma(s)} \int_{-\infty}^{x} (x - t)^{s - 1}f(t) \,\mathrm{d}t = \, \left(_{-\infty}\,\mathrm{I}\,_{x}^{s}f\right)(x)$$
            if all conditions of Lemma \ref{Cauchy's Formula with Improper Bounds} are satisfied.
            
        \end{proof}

\section{An alternate representation of the Riemann zeta function}\label{sec3}

    In order to define the Riemann zeta function as a fractional integral, we first consider the Banach spaces of exponential type functions $X_{\sigma} = L^{1}(e^{-\sigma|t|}dt)$ \cite{Stein}.

    \begin{lemma} 
    \label{Analysis of Exponentials}
    
        Let $f(t) = \frac{1}{e^{-t} + 1}$ be an exponential function of type $\sigma = 1$ with $\lVert f \rVert = 1$. Then $\left(_{-\infty}\,\mathrm{I}\,_{x}^{s}f\right)(x)$ is convergent for $\Re(s) > 0$ and $\forall x \in (-\infty, 0]$.
        
    \end{lemma}

        \begin{proof}

            Let $Y = (\mathbb{C}, |\cdot|)$ such that $_{-\infty}\,\mathrm{I}\,_{x}^{s} : X_{\sigma} \longrightarrow Y$, $\Re(s) > 0$, $x \in (-\infty, 0]$ defines a linear operator for all $f \in X_{\sigma}$ which satisfy the criteria of Lemma \ref{Cauchy's Formula with Improper Bounds} and Corollary \ref{Cauchy's Formula Extended using RLI}. We endow this operator with the norm
            $$\lVert _{-\infty}\,\mathrm{I}\,_{x}^{s} \rVert = \sup\left\{\frac{\left|_{-\infty}\,\mathrm{I}\,_{x}^{s}f\right|}{\lVert f \rVert} \middle| f \in X_{\sigma}, f \neq 0\right\}.$$
            Now, let $f(t) = \frac{1}{e^{-t} + 1}$ for real $t$. Since $0 < f(t) < 1 \; \forall t \in \mathbb{R}$ and, subsequently,
            $$\limsup_{|t| \to \infty} \frac{\log\left(f(t)\right)}{t} = 1 = \sigma,$$
            it is true that $f \in X_1$ with
            $$\lVert f \rVert = \int_{-\infty}^{\infty} f(t)e^{-|t|} \,\mathrm{d}t.$$
            Since
            $$1 - f(t) = 1 - \frac{1}{e^{-t} + 1} = \frac{e^{-t}}{e^{-t} + 1} = \frac{1}{e^{t} + 1} = f(-t),$$
            we write
            $$\lVert f \rVert = \int_{0}^{\infty} (f(t)+f(-t))e^{-|t|}\,\mathrm{d}t = \int_{0}^{\infty}e^{-|t|}\,\mathrm{d}t = 1.$$
            Thus,
            $$\lVert _{-\infty}\,\mathrm{I}\,_{x}^{s}f \rVert = \left|_{-\infty}\,\mathrm{I}\,_{x}^{s}f\right| \leq \frac{1}{|\Gamma(s)|} \int_{-\infty}^{x} \left|(x - t)^{s - 1}\right|f(t)\,\mathrm{d}t$$
            $$\leq \frac{1}{|\Gamma(s)|} \int_{-\infty}^{x} (x - t)^{\Re(s) - 1}f(t)\,\mathrm{d}t$$
            $$= \frac{1}{|\Gamma(s)|} \int_{-x}^{\infty} (t + x)^{\Re(s) - 1}f(-t)\,\mathrm{d}t.$$
            From the definition of an incomplete $\eta$ function given in Equation \ref{Incomplete eta} and as a consequence of the integral representation of $\eta$ given in Equation \ref{Zeta as eta Integral}, we have
            $$\frac{1}{|\Gamma(s)|} \int_{-x}^{\infty} (t + x)^{\Re(s) - 1}f(-t)\,\mathrm{d}t$$
            $$= \frac{\Gamma(s)}{|\Gamma(s)|} \eta(\Re(s), -x) \leq \eta(\Re(s)).$$
            Since the complete $\eta$ function is well defined and convergent for $\Re(s) > 0$, the integral resulting from $\left(_{-\infty}\,\mathrm{I}\,_{x}^{s}f\right)(x)$ is also convergent $\forall x \in (-\infty, 0]$.
            
        \end{proof}
    
    Now, by Equations \ref{Incomplete eta} and \ref{Incomplete zeta} in Section \ref{sec1}, an incomplete Riemann $\zeta$ function in the positive complex half-plane is defined as
    $$\zeta(s, x) = \frac{2^{s}}{2^{s} - 2} \eta(s, x) = \frac{2^{s}}{(2^{s} - 2) \Gamma(s)} \int_{x}^{\infty} \frac{(t - x)^{s - 1}}{e^{t} + 1} \,\mathrm{d}t, \quad x \in [0, \infty)$$
    where $\zeta(s, 0) = \zeta(s)$.
    Consequently, this incomplete Riemann $\zeta$ function, convergent for $\Re(s) > 0$ except at $s = 1$ and $x \in [0, \infty)$, can be represented using a lower Riemann-Liouville fractional integral as in Corollary \ref{Cauchy's Formula Extended using RLI}.

    \begin{definition} 
    \label{Incomplete zeta Fractional Integral}
    
        Let $f(t) = \frac{1}{e^{-t} + 1}$. Then for $\Re(s) > 0$ and $x \in (-\infty, 0]$,
        $$ \zeta(s, x) = \frac{2^{s}}{2^{s} - 2} \,\left(_{-\infty}\,\mathrm{I}\,_{x}^{s}f\right)(x).$$
        
    \end{definition}

        \begin{proof}

            Let $f \in X_1 = L^{1}(e^{-|t|}\mathrm{d}t)$ be defined as $f(t) = \frac{1}{e^{-t} + 1}$. Then the fractional integral of order $s$ of $f$ with $\Re(s) > 0$ and upper bound $x \in (-\infty, 0]$ is given as
            $$\,\left(_{-\infty}\,\mathrm{I}\,_{x}^{s}f\right)(x) = \frac{1}{\Gamma(s)} \int_{-\infty}^{x} \frac{(x - t)^{s - 1}}{e^{-t} + 1} \,\mathrm{d}t = \frac{1}{\Gamma(s)} \int_{-x}^{\infty} \frac{(x + t)^{s - 1}}{e^{t} + 1} \,\mathrm{d}t = \eta(s, x)$$
            by Equation \ref{Incomplete eta}. By Lemma \ref{Analysis of Exponentials}, it is known that this integral is convergent for $\Re(s) > 0$ and $x \in (-\infty, 0]$, hence satisfying the criteria of Lemma \ref{Cauchy's Formula with Improper Bounds} and, further, Corollary \ref{Cauchy's Formula Extended using RLI}. Now, multiplying by $\frac{2^{s}}{2^{s} - 2}$, we get
            $$\frac{2^{s}}{2^{s} - 2} \eta(s, x) = \zeta(s, x)$$
            by Equation \ref{Incomplete zeta}. It then follows that $\zeta(s, x)$ is also convergent for $\Re(s) > 0$ and $x \in (-\infty, 0]$ except for the pole introduced at $s = 1$.
            
        \end{proof}

    Finally, by evaluating the incomplete $\zeta$ function defined above, we obtain the following relation between the Riemann $\zeta$ function and a lower-bounded Riemann-Liouville integral:
    \begin{equation}\label{Zeta Fractional Integral}
        \zeta(s, 0) = \frac{2^{s}}{2^{s} - 2} \,\left(_{-\infty}\,\mathrm{I}\,_{0}^{s}f\right)(0) = \zeta(s).
    \end{equation}
    In further sections, the authors discuss the properties of this new representation for $\zeta$ and the implications thereof.

\section{Implications}\label{sec4}

    Writing an incomplete Riemann $\zeta$ function using a lower-bounded Riemann-Liouville integral, as in Definition \ref{Incomplete zeta Fractional Integral}, has some potential advantages over other forms for $\Re(s) > 0$. One such advantage is found in the semigroup property associated with Riemann-Liouville integrals \cite{Riesz}.

    \begin{theorem} 
    \label{Semi-Group Property for RLI}

        For any two complex numbers $\alpha$ and $\beta$ such that $\Re(\alpha) > 0$ and $\Re(\beta) > 0$, if $\,\left(_{-\infty}\,\mathrm{I}\,_{x}^{\alpha}f\right)(x)$, $\,\left(_{-\infty}\,\mathrm{I}\,_{x}^{\beta}f\right)(x)$, and $\,\left(_{-\infty}\,\mathrm{I}\,_{x}^{\alpha + \beta}f\right)(x)$ are all convergent, then

        $$\,\left(_{-\infty}\,\mathrm{I}\,_{x}^{\alpha + \beta}f\right)(x) = \,\left(_{-\infty}\,\mathrm{I}\,_{x}^{\alpha}\right)\,\left(_{-\infty}\,\mathrm{I}\,_{x}^{\beta}f\right)(x) = \,\left(_{-\infty}\,\mathrm{I}\,_{x}^{\beta}\right)\,\left(_{-\infty}\,\mathrm{I}\,_{x}^{\alpha}f\right)(x).$$
    
    \end{theorem}

        \begin{proof}

            Consider two improper Riemann-Liouville integrals of the continuous function $f$, which are convergent for some $x$ and whose orders are $\alpha$ and $\beta$, respectively, such that $\Re(\alpha) > 0$ and $\Re(\beta) > 0$. These are written
            $$\,\left(_{-\infty}\,\mathrm{I}\,_{x}^{\alpha}f\right)(x) \quad \text{and} \quad \left(_{-\infty}\,\mathrm{I}\,_{x}^{\beta}f\right)(x).$$
            Then
            $$\,\left(_{-\infty}\,\mathrm{I}\,_{x}^{\alpha}\right)\,\left(_{-\infty}\,\mathrm{I}\,_{x}^{\beta}f\right)(x) = \frac{1}{\Gamma(\alpha)} \int_{-\infty}^{x} (x - t)^{\alpha - 1}\,\left(_{-\infty}\,\mathrm{I}\,_{t}^{\beta}f\right)(t) \,\mathrm{d}t$$
            $$= \frac{1}{\Gamma(\alpha)\Gamma(\beta)} \int_{-\infty}^{x} \int_{-\infty}^{t} (x - t)^{\alpha - 1}(t - s)^{\beta - 1}f(s) \,\mathrm{d}s\,\mathrm{d}t$$
            $$= \frac{1}{\Gamma(\alpha)\Gamma(\beta)} \int_{-\infty}^{x} f(s) \int_{s}^{x} (x - t)^{\alpha - 1}(t - s)^{\beta - 1} \,\mathrm{d}t\,\mathrm{d}s.$$
            Changing variables from $t$ to $r$ where $t = s + (x - s)r$ yields
            $$\frac{1}{\Gamma(\alpha)\Gamma(\beta)} \int_{-\infty}^{x} (x - s)^{\alpha + \beta - 1}f(s) \int_{0}^{1} (1 - r)^{\alpha - 1}r^{\beta - 1} \,\mathrm{d}r\,\mathrm{d}s.$$
            The inner integral is the Beta function $B(\alpha, \beta) = \frac{\Gamma(\alpha)\Gamma(\beta)}{\Gamma(\alpha + \beta)}$. Substituting this gives
            $$\frac{1}{\Gamma(\alpha + \beta)} \int_{-\infty}^{x} (x - s)^{\alpha + \beta - 1}f(s) \,\mathrm{d}s = \,_{-\infty}\,\mathrm{I}\,_{x}^{\alpha + \beta}f(x).$$
            Interchanging $\alpha$ and $\beta$ in this proof yields identical results due to the commutativity of addition and multiplication. Therefore,
            $$\,\left(_{-\infty}\,\mathrm{I}\,_{x}^{\alpha}\right)\,\left(_{-\infty}\,\mathrm{I}\,_{x}^{\beta}f\right)(x) = \,\left(_{-\infty}\,\mathrm{I}\,_{x}^{\beta}\right)\,\left(_{-\infty}\,\mathrm{I}\,_{x}^{\alpha}f\right)(x) = \,\left(_{-\infty}\,\mathrm{I}\,_{x}^{\alpha + \beta}f\right)(x).$$
            
        \end{proof}

    Applying this property to the incomplete zeta function defined in Definition \ref{Incomplete zeta Fractional Integral}, which is convergent for $\Re(s) > 0$ except at $s = 1$ and $\forall x \in (-\infty, 0]$, we get
    \begin{equation}\label{Zeta Semigroup Property}
        \begin{split}
            \zeta(\alpha + \beta, x) &= \frac{2^{\alpha+\beta}}{2^{\alpha+\beta}-2}\,\left(_{-\infty}\,\mathrm{I}\,_{x}^{\alpha}\right)\left(_{-\infty}\,\mathrm{I}\,_{x}^{\beta}f\right)(x) \\
            &= \frac{2^{\alpha+\beta}}{2^{\alpha+\beta}-2}\,\left(_{-\infty}\,\mathrm{I}\,_{x}^{\beta}\right)\left(_{-\infty}\,\mathrm{I}\,_{x}^{\alpha}f\right)(x)
        \end{split}
    \end{equation}
    for $\Re(\alpha), \Re(\beta) > 0$. Further utilizing Definition \ref{Zeta Fractional Integral}, we can define $\zeta(s, x)$ recursively in terms of its partial derivatives at nearby inputs.

    \begin{theorem} 
    \label{Functional Equation with Derivatives}

        For all $x \in (-\infty, 0]$ and $\Re(s) > 0$ except at $s = 1$,
        $$\zeta(s, x) = \frac{2^{s} - 1}{2^{s} - 2} \frac{\mathrm{\partial}}{\mathrm{\partial}{x}} \zeta(s + 1, x).$$
        
    \end{theorem}

        \begin{proof}

            By Definition \ref{Incomplete zeta Fractional Integral}, the  incomplete Riemann $\zeta$ function is convergent for $\Re(s) > 0$, except at $s = 1$, and $\forall x \in (-\infty, 0]$. Furthermore, it can be represented as
            $$ \zeta(s, x) = \frac{2^{s}}{2^{s} - 2} \,\left(_{-\infty}\,\mathrm{I}\,_{x}^{s}f\right)(x).$$
            Then clearly,
            $$ \zeta(s + 1, x) = \frac{2^{s + 1}}{2^{s + 1} - 2} \,\left(_{-\infty}\,\mathrm{I}\,_{x}^{s + 1}f\right)(x).$$
            Differentiation yields
            $$\frac{\mathrm{d}}{\mathrm{d}x} \left[ \frac{2^{s + 1}}{2^{s + 1} - 2} \,\left(_{-\infty}\,\mathrm{I}\,_{x}^{s + 1}f\right)(x) \right]$$
            $$= \frac{2^{s + 1}}{(2^{s + 1} - 2)\Gamma(s + 1)} \frac{\mathrm{d}}{\mathrm{d}x} \int_{-\infty}^{x} (x - t)^{s} \frac{1}{e^{-t} + 1} \,\mathrm{d}t.$$
            Applying the Leibniz Integral Rule as in Lemma \ref{Cauchy's Formula with Improper Bounds}, we have
            $$\frac{2^{s + 1}}{(2^{s + 1} - 2)\Gamma(s + 1)} \int_{-\infty}^{x} \frac{\mathrm{\partial}}{\mathrm{\partial}{x}} (x - t)^{s} \frac{1}{e^{-t} + 1} \,\mathrm{d}t$$
            $$= \frac{2^{s + 1}}{(2^{s + 1} - 2)\Gamma(s)} \int_{-\infty}^{x} (x - t)^{s - 1} \frac{1}{e^{-t} + 1} \,\mathrm{d}t$$
            $$= \frac{2^{s + 1}}{2^{s + 1} - 2} \,\left(_{-\infty}\,\mathrm{I}\,_{x}^{s}f\right)(x)$$
            $$= \frac{2^{s} - 2}{2^{s} - 1} \left[ \frac{2^{s}}{2^{s} - 2} \,\left(_{-\infty}\,\mathrm{I}\,_{x}^{s}f\right)(x) \right]$$
            $$= \frac{2^{s} - 2}{2^{s} - 1} \zeta(s, x)$$
            which implies that
            $$\zeta(s, x) = \frac{2^{s} - 1}{2^{s} - 2} \frac{\mathrm{\partial}}{\mathrm{\partial}{x}} \zeta(s + 1, x).$$
            
        \end{proof}

\section{Conclusion}\label{sec5}

    In this paper, the authors defined a relationship between an incomplete Riemann $\zeta$ function and a corresponding lower-bounded Riemann-Liouville fractional integral. From Equations \ref{Incomplete eta} and \ref{Incomplete zeta} and Definition \ref{Incomplete zeta Fractional Integral}, we get
    \begin{equation}\label{Conclusion Equation 1}
        \zeta(s, x) = \frac{2^{s}}{2^{s} - 2} \eta(s, x) = \frac{2^{s}}{2^{s} - 2} \,\left(_{-\infty}\,\mathrm{I}\,_{x}^{s}f\right)(x).
    \end{equation}
    Furthermore, we can relate these incomplete versions of the Riemann $\zeta$ function and Dirichlet $\eta$ function by evaluating the equation above at $x = 0$. This follows directly from Equations \ref{Zeta Equals eta} and \ref{Zeta Fractional Integral}:
    \begin{equation}\label{Conclusion Equation 2}
        \zeta(s) = \frac{2^{s}}{2^{s} - 2} \eta(s) = \frac{2^{s}}{2^{s} - 2} \,\left(_{-\infty}\,\mathrm{I}\,_{0}^{s}f\right)(0).
    \end{equation}
    
    These equations endow an incomplete Riemann $\zeta$ function with two properties that other representations lack. First, we have the semigroup property of fractional integrals from Theorem \ref{Semi-Group Property for RLI}, which applies to $\zeta(s, x)$ $\forall x \in (-\infty, 0]$ as in Equation \ref{Zeta Semigroup Property}, where $s = \alpha + \beta$ and $\Re(\alpha), \Re(\beta) > 0$. Additionally, we have Theorem \ref{Functional Equation with Derivatives}, which defines a functional relationship for $\zeta(s, x)$ for $\Re(s) > 0$ and $\forall x \in (-\infty, 0]$ through partial differentiation with respect to $x$. Together, these properties represent a unique flexibility in an incomplete Riemann $\zeta$ function that may be useful in analyzing the properties of the complete $\zeta$ function. For instance, the authors are interested by the potential symmetries, or lack thereof, present in this incomplete $\zeta$ function as it approaches the values of the complete $\zeta$ function. It may be possible examine such symmetry for variables values of $s$ and $x$ as follows:
    \begin{equation}\label{Conclusion Equation 3}
        \begin{split}
            \zeta(s_1, x) = \zeta(s_2, x), \quad \forall x \in (-\infty, 0], \quad \Re(s_1), \Re(s_2) > 0 \\
            \Rightarrow \frac{2^{s_1}}{2^{s_1} - 2} \,\left(_{-\infty}\,\mathrm{I}\,_{x}^{s_1}f\right)(x) = \frac{2^{s_2}}{2^{s_2} - 2} \,\left(_{-\infty}\,\mathrm{I}\,_{x}^{s_2}f\right)(x).
        \end{split}
    \end{equation}
    For this reason, the authors suggest further development of Equations \ref{Conclusion Equation 1} and \ref{Conclusion Equation 2} using the properties and utilities defined in Theorems \ref{Semi-Group Property for RLI} and \ref{Functional Equation with Derivatives} in order to better understand the complete Riemann $\zeta$ function and the conjectures surrounding it.

\bibliographystyle{amsplain}

\end{document}